\documentclass[11pt,letterpaper]{amsart}

\usepackage{amssymb}
\usepackage{enumerate}
\usepackage{bbm}

\newtheorem{theorem}{Theorem}
\newtheorem{proposition}[theorem]{Proposition}

\newtheorem{lemma}[theorem]{Lemma}
\theoremstyle{definition}
\newtheorem*{remark}{Remark}

\newcommand{\FF}{\mathbb{F}}
\newcommand{\CC}{\mathbb{C}}
\newcommand{\RR}{\mathbb{R}}
\newcommand{\NN}{\mathbb{N}}
\newcommand{\QQ}{\mathbb{Q}}

\newcommand{\1}{\mathbbm{1}}
\newcommand{\e}{\epsilon}

\newcommand{\norm}[1]{\lVert#1\rVert}

\newcommand{\abs}[1]{\lvert#1\rvert}
\newcommand{\bigabs}[1]{\big\lvert#1\big\rvert}
\newcommand{\biggabs}[1]{\bigg\lvert#1\bigg\rvert}
\newcommand{\Biggabs}[1]{\Bigg\lvert#1\Bigg\rvert}

\DeclareMathOperator{\ord}{ord}
\DeclareMathOperator{\Tr}{Tr}
\DeclareMathOperator{\N}{N}

\begin{document}

\title{Asymptotically optimal Boolean functions}

\author{Kai-Uwe Schmidt}
\address{Department of Mathematics, Paderborn University, Warburger Str.\ 100, 33098 Paderborn, Germany.}
\email[K.-U. Schmidt]{kus@math.upb.de}

\thanks{K.-U. Schmidt is partly supported by German Research Foundation (DFG)}

\date{22 November 2017}


\begin{abstract}
The largest Hamming distance between a Boolean function in $n$ variables and the set of all affine Boolean functions in $n$ variables is known as the covering radius $\rho_n$ of the $[2^n,n+1]$ Reed-Muller code. This number determines how well Boolean functions can be approximated by linear Boolean functions. We prove that 
\[
\lim_{n\to\infty}2^{n/2}-\rho_n/2^{n/2-1}=1,
\]
which resolves a conjecture due to Patterson and Wiedemann from~1983. 
\end{abstract}

\maketitle

\thispagestyle{empty}


\section{Introduction and results}

The Hamming distance of two Boolean functions $F,G:\FF_2^n\to\FF_2$ is
\[
d(F,G)=\#\{y\in\FF_2^n:F(y)\ne G(y)\}.
\]
Put
\[
\rho_n=\max_F\min_Gd(F,G),
\]
where the maximum is over all functions $F$ from $\FF_2^n$ to $\FF_2$ and the minimum is over all $2^{n+1}$ affine functions $G$ from $\FF_2^n$ to $\FF_2$. Then~$\rho_n$ equals the covering radius of the $[2^n,n+1]$ Reed-Muller code, whose determination is one of the oldest and most difficult open problems in coding theory~\cite{HelKloMyk1978},~\cite{PatWie1983},~\cite{Slo1987}. We refer to~\cite{CohHonLitLob1997} for background on the covering radius of codes in general and its combinatorial and geometric significance. The determination of $\rho_n$ also answers the question of how well Boolean functions can be approximated by linear functions, which is of significance in cryptography~\cite{Car2010}. One can also interprete~$\rho_n$ in terms of the Fourier coefficients of Boolean functions (see Section~\ref{sec:proof}).
\par
It is convenient to define
\[
\mu_n=2^{n/2}-\rho_n/2^{n/2-1}.
\]
An averaging argument shows that $\mu_n\ge 1$ (see Section~\ref{sec:proof}) and a simple recursive construction involving functions of the form $F(y)+uv$ on $\FF_2^{n+2}$ shows that $\mu_{n+2}\le\mu_n$. The fact that $\mu_2=1$ implies that $\mu_n=1$ for all even $n$; the functions attaining the minimum are known as \emph{bent} functions and these have been studied extensively for more than forty years~\cite{Rot1976},~\cite{Mes2016}.
\par
We are interested in the case that $n$ is odd. Since $\mu_1=\sqrt{2}$, we have $1\le\mu_n\le\sqrt{2}$. It is known that equality holds in the upper bound for $n=3$ (trivial), for $n=5$~\cite{BerWel1972}, and for $n=7$~\cite{Myk1980},~\cite{Hou1996}. It was suggested in~\cite{HelKloMyk1978} that $\mu_n=\sqrt{2}$ for all odd $n$, which was disproved by Patterson and Wiedemann~\cite{PatWie1983}, by showing that
\begin{equation}
\mu_n\le \sqrt{729/512}= 1.19\dots\quad\text{for each $n\ge 15$}.   \label{eqn:mu_upper_bounds}
\end{equation}
More recently it was shown by Kavut and Y\"ucel~\cite{KavYuc2010} that
\[
\mu_n\le \sqrt{49/32}=1.23\dots\quad\text{for each $n\ge 9$}.
\]
Patterson and Wiedemann~\cite{PatWie1983} also conjectured that $\lim_{n\to\infty}\mu_n=1$. However no improvement of~\eqref{eqn:mu_upper_bounds} for large $n$ has been found since this conjecture has been posed in 1983. We shall prove that this conjecture is true.
\par
\begin{theorem}
\label{thm:pw}
We have $\lim_{n\to\infty}\mu_n=1$.
\end{theorem}
\par
Several researchers (for example in~\cite{SebZhaZhe1995},~\cite{Dob1995},~\cite{MaiSar2002}) also investigated 
\[
\rho'_n=\max_F\min_Gd(F,G),
\]
where now the maximum is over all \emph{balanced} functions $F$ from $\FF_2^n$ to $\FF_2$ (which means that $F$ takes the values $0$ and $1$ equally often) and the minimum is still over all affine functions $G$ from $\FF_2^n$ to $\FF_2$. Put
\[
\mu'_n=2^{n/2}-\rho'_n/2^{n/2-1}.
\]
Slight modifications of our proof of Theorem~\ref{thm:pw} lead to the following result, which proves a conjecture due to Dobbertin~\cite[Conjecture~B]{Dob1995} from 1995.
\begin{theorem}
\label{thm:pw_balanced}
We have $\lim_{n\to\infty}\mu'_n=1$.
\end{theorem}


\section{Proof of main result}
\label{sec:proof}

In what follows, we identify $\FF_2^n$ with $\FF_{2^n}$ and consider functions $f:\FF_{2^n}\to\CC$. Let $\psi:\FF_{2^n}\to\CC$ be the canonical additive character of $\FF_{2^n}$, which is given by $\psi(y)=(-1)^{\Tr(y)}$, where $\Tr$ is the absolute trace function on $\FF_{2^n}$. The \emph{Fourier transform} of $f$ is the function $\hat f:\FF_{2^n}\to\CC$ given by
\[
\hat f(a)=\frac{1}{2^{n/2}}\sum_{y\in \FF_{2^n}}f(y)\psi(ay).
\]
It is well known~\cite{Car2010} and readily verified that
\[
\mu_n=\min_f\max_{a\in\FF_{2^n}}\;\abs{\hat f(a)},
\]
where the minimum is over all functions $f:\FF_{2^n}\to \{-1,1\}$. From Parseval's identity
\[
\sum_{a\in\FF_{2^n}}\abs{\hat f(a)}^2=\sum_{y\in\FF_{2^n}}\abs{f(y)}^2
\]
it follows now that $\mu_n\ge 1$.
\par
We shall construct functions $f$ with image $\{-1,1\}$ for which $\abs{\hat f(a)}$ is small for all $a\in\FF_{2^n}$. Let $H$ be a (multiplicative) subgroup of $\FF_{2^n}^*$ of index~$v$ and define the indicator function of $H$ on $\FF_{2^n}$ by
\[
\1_H(y)=\begin{cases}
1 & \text{for $y\in H$}\\
0 & \text{otherwise}.
\end{cases}
\]
Let $h:H\to\{-1,1\}$ be a function to be specified later. Let $T$ be a complete system of coset representatives of $H$ in $\FF_{2^n}^*$ and let $g:T\to \{0,-1,1\}$ be a function satisfying $g(z)=0$ if and only if $z\in H$ and such that $g$ is balanced, which means that the images $-1$ and $1$ occur equally often. We define $f:\FF_{2^n}\to\{-1,1\}$ by $f(0)=1$ and
\[
f(y)=\1_H(y)\,h(y)+\sum_{z\in T}\1_H(y/z)\,g(z)\quad\text{for $y\in\FF_{2^n}^*$}.
\]
Note that $f$ is constant on the cosets of $H$, except for $H$ itself. Such functions were also used by Patterson and Wiedemann~\cite{PatWie1983} in their proof of~\eqref{eqn:mu_upper_bounds} and have been also studied in several other papers, for example in~\cite{BriGilLan2005}.
\par
Recall that $\ord_v(a)$ for integers $v$ and $a$ with $v>0$ and $\gcd(a,v)=1$ is the smallest positive integer $t$ such that $a^t\equiv 1\pmod v$.
\begin{proposition}
\label{pro:main}
Let $e$ be a positive integer and let $v=7^e$. Then there exists an odd multiple $n$ of $\ord_v(2)$ and a function $h:H\to\{-1,1\}$ such that the function $f:\FF_{2^n}\to\{-1,1\}$, defined above, satisfies
\[
\max_{a\in \FF_{2^n}}\;\abs{\hat f(a)}\le 1+12\sqrt{\frac{\log(2v)}{v}}.
\]
\end{proposition}
\par
A routine induction shows that $\ord_{7^e}(2)$ equals $\phi(7^e)/2=3\cdot 7^{e-1}$, and so is odd, for all positive integers~$e$ (where $\phi$ is the Euler totient function). Now let $e$ tend to infinity in Proposition~\ref{pro:main} and use $\mu_n=1$ for all even~$n$ and the inequality $1\le \mu_{n+2}\le\mu_n$ for all $n$ to deduce Theorem~\ref{thm:pw} from Proposition~\ref{pro:main}.
\par
\begin{remark}
Proposition~\ref{pro:main} remains true if $7$ is replaced by an arbitrary prime $q\equiv 3\pmod 4$ such that $\ord_{q^e}(2)=\phi(q^e)/2$ for each $e\in\{1,2\}$ (which ensures that this identity holds for all positive integers $e$). The first primes of this form are $7,23,47,71,79$, but it is not known whether there are infinitely many such primes. We choose $q=7$ to keep our proof simple. 
\end{remark}
\par
To prove Proposition~\ref{pro:main}, we define functions $f_1,f_2:\FF_{2^n}\to\{0,-1,1\}$ by
\begin{align*}
f_1(y)&=\1_H(y)\,h(y),\\
f_2(y)&=\sum_{z\in T}\1_H(y/z)\,g(z),
\end{align*}
so that $f(y)=f_1(y)+f_2(y)$ for all $y\in\FF_{2^n}^*$ and $\hat f(a)=2^{-n/2}+\hat f_1(a)+\hat f_2(a)$ for all $a\in\FF_{2^n}$. We shall see that bounding $\abs{\hat f_1(a)}$ is not difficult using known results from probabilistic combinatorics. Bounding $\abs{\hat f_2(a)}$ requires a little more work. 
\par
For a multiplicative character $\chi$ of $\FF_{2^n}$, the \emph{Gauss sum} $G(\chi)$ is defined to be
\[
G(\chi)=\sum_{y\in \FF_{2^n}^*}\psi(y)\chi(y).
\]
It is well known that $\abs{G(\chi)}=2^{n/2}$ if $\chi$ is nontrivial (which means that $\chi(y)\ne 1$ for some $y\in\FF_{2^n}^*$)~\cite[Theorem 5.11]{LidNie1997}.
\par
We begin with the following elementary lemma.
\par
\begin{lemma}
\label{lem:gauss_sum_fourier}
Let $\e>0$ and suppose that, for all nontrivial multiplicative characters $\chi$ of $\FF_{2^n}$ of order dividing $v$, we have
\[
\biggabs{\frac{G(\chi)}{2^{n/2}}-1}\le\e.
\]
Then we have
\[
\max_{a\in\FF_{2^n}}\;\abs{\hat f_2(a)}\le 1+\e\,v.
\]
\end{lemma}
\begin{proof}
Since $g$ is balanced, we have $\hat f_2(0)=0$, so let $a\in\FF_{2^n}^*$. Let $\chi$ be a multiplicative character of $\FF_{2^n}$ of order $v$. Then the indicator function $\1_H$ satisfies
\begin{equation}
\1_H(y)=\frac{1}{v}\sum_{j=0}^{v-1}\chi^j(y)\quad\text{for each $y\in \FF_{2^n}^*$}.   \label{eqn:Fourier_indicator}
\end{equation}
Therefore we have
\begin{align*}
\sum_{y\in \FF_{2^n}}\1_H(y)\psi(ay)&=\frac{1}{v}\sum_{j=0}^{v-1}\sum_{y\in \FF_{2^n}^*}\psi(ay)\chi^j(y)\\
&=\frac{1}{v}\sum_{j=0}^{v-1}\chi^j(a^{-1})\sum_{y\in\FF_{2^n}^*}\psi(y)\chi^j(y)\\
&=\frac{1}{v}\sum_{j=0}^{v-1}\overline{\chi}^j(a)G(\chi^j),
\end{align*}
which we use to obtain
\begin{align*}
2^{n/2}\hat f_2(a)&=\sum_{y\in\FF_{2^n}}\sum_{z\in T}\1_H(y/z)g(z)\psi(ay)\\
&=\sum_{z\in T}g(z)\sum_{y\in \FF_{2^n}}\1_H(y)\psi(ayz)\\
&=\frac{1}{v}\sum_{z\in T}g(z)\sum_{j=0}^{v-1}\overline{\chi}^j(az)\,G(\chi^j)\\
&=\frac{1}{v}\sum_{j=0}^{v-1}G(\chi^j)\,\overline{\chi}^j(a)\sum_{z\in T}g(z)\,\overline{\chi}^j(z).
\end{align*}
Now write $G(\chi^j)=2^{n/2}(1+\gamma_j)$, so that $\abs{\gamma_j}\le \e$ for all $j\in\{1,\dots,v-1\}$ by our assumption. Since $G(\chi^0)=-1$, we obtain $\hat f_2(a)=M(a)+E(a)$, where
\begin{align*}
M(a)&=\frac{1}{v}\sum_{j=1}^{v-1}\overline{\chi}^j(a)\sum_{z\in T}g(z)\overline{\chi}^j(z)-\frac{1}{2^{n/2}\,v}\sum_{z\in T}g(z)\\
&=\frac{1}{v}\sum_{z\in T}g(z)\sum_{j=1}^{v-1}\overline{\chi}^j(az)\\
&=\frac{1}{v}\sum_{z\in T}g(z)\sum_{j=0}^{v-1}\overline{\chi}^j(az)-\frac{1}{v}\sum_{z\in T}g(z)\\
&=g(b)\quad\text{for $b\in T$ such that $ab\in H$},
\end{align*}
using that $g$ is balanced and~\eqref{eqn:Fourier_indicator} again, and
\[
\abs{E(a)}=\Biggabs{\frac{1}{v}\sum_{j=1}^{v-1}\gamma_j\,\overline{\chi}^j(a)\sum_{z\in T}g(z)\overline{\chi}^j(z)}\le \e\,v.
\]
This gives the required result.
\end{proof}
\par
The following explicit evaluation of certain Gauss sums~\cite[Proposition~4.2]{Lan1997} (see also~\cite[Theorem~4.1]{YanXia2010}) will help us to control the error term in Lemma~\ref{lem:gauss_sum_fourier}.
\begin{lemma}[{\cite[Proposition~4.2]{Lan1997}}]
\label{lem:gauss_sum_index_two}
Let $q>3$ be a prime satisfying $q\equiv 3\pmod 4$. Let $d$ be a positive integer, write $k=\phi(q^d)/2$, and let $p$ be a prime such that $\ord_{q^d}(p)=k$. Let $\tau$ be a multiplicative character of $\FF_{p^k}$ of order~$q^d$ and let $h$ be the class number of $\QQ(\sqrt{-q})$. Then
\[
G(\tau)=\frac{1}{2}(a+b\sqrt{-q})p^{(k-h)/2},
\]
where $a$ and $b$ are integers satisfying $a,b\not\equiv 0\pmod p$, $a^2+b^2q=4p^h$, and $ap^{(k-h)/2}\equiv -2\pmod q$.
\end{lemma}
\par
We shall apply Lemma~\ref{lem:gauss_sum_index_two} with $p=2$ and $q=7$. Since the class number of $\QQ(\sqrt{-7})$ equals $1$ and
\[
2^{(\phi(7^d)/2-1)/2}\equiv 2\pmod 7
\]
for all positive integers $d$, we find that $a=-1$ and $b^2=1$ in this case. As noted after Proposition~\ref{pro:main}, we have $\ord_{7^d}(2)=\phi(7^d)/2$ for all positive integers~$d$, so that the hypothesis in Lemma~\ref{lem:gauss_sum_index_two} is satisfied for $p=2$ and~$q=7$.
\par
As a corollary to Lemma~\ref{lem:gauss_sum_index_two}, we obtain the following lemma.
\begin{lemma}
\label{lem:Gauss_sum_evaluations}
Let $e$ and $d$ be integers satisfying $1\le d\le e$ and write $m=\ord_{7^e}(2)$. Let $\chi$ be a multiplicative character of $\FF_{2^{sm}}$ of order $7^d$. Then
\[
\frac{G(\chi)}{2^{sm/2}}=-(-1)^s\left(\frac{-1\pm\sqrt{-7}}{2^{3/2}}\right)^{7^{e-d}\,s},
\]
where the sign depends on $\chi$.
\end{lemma}
\begin{proof}
Write $k=\ord_{7^d}(2)$ and let $\tau$ be the multiplicative character of $\FF_{2^k}$ such that $\chi$ is the lifted character of $\tau$, by which we mean that $\chi=\tau\circ\N$, where $\N$ is the field norm from $\FF_{2^{sm}}$ to $\FF_{2^k}$. Lemma~\ref{lem:gauss_sum_index_two} and the preceding discussion implies that
\[
G(\tau)=2^{(k-3)/2}(-1\pm\sqrt{-7}).
\]
From the Davenport-Hasse theorem~\cite[Theorem 5.14]{LidNie1997} we find that
\[
G(\chi)=-(-1)^{sm/k}\left[2^{(k-3)/2}(-1\pm\sqrt{-7})\right]^{sm/k},
\]
and the lemma follows since $m/k=\phi(7^e)/\phi(7^d)=7^{e-d}$.
\end{proof}
\par
The next lemma gives the desired control for the error term in Lemma~\ref{lem:gauss_sum_fourier}.
\begin{lemma}
\label{lem:Gauss_sum_error}
Let $e$ be a positive integer, let $v=7^e$, and write $m=\ord_v(2)$. Let $\e>0$. Then there is an infinite set $S$ of odd positive integers such that, for all $s\in S$ and all nontrivial multiplicative characters $\chi$ of $\FF_{2^{sm}}$ of order dividing $v$, we have
\[
\abs{\arg G(\chi)}\le \e.
\]
Here, $\arg(\xi)\in(-\pi,\pi]$ is the principal angle of a nonzero complex number~$\xi$.
\end{lemma}
\begin{proof}
Let $\tau$ be a multiplicative character of $\FF_{2^m}$ of order $v$. Since the units of the ring of algebraic integers in $\QQ(\sqrt{-7})$ are $\pm 1$, we find from Lemma~\ref{lem:Gauss_sum_evaluations} that $G(\tau)/2^{m/2}$ is not a root of unity. Therefore Weyl's uniform distribution theorem~\cite[Satz~2]{Wey1916} implies that $([G(\tau)/2^{m/2}]^{2i})_{i\in\NN}$, and therefore also $([G(\tau)/2^{m/2}]^{2i+1})_{i\in\NN}$, is uniformly distributed on the complex unit circle. Hence there is an infinite set $S$ of odd positive integers such that
\[
\abs{\arg(G(\tau)^s)}\le \frac{\e}{7^{e-1}}
\]
for all $s\in S$.
\par
Let $s\in S$ and let $\tau'$ be the lifted character of $\tau$ to $\FF_{2^{sm}}$, namely $\tau'=\tau\circ\N$, where $\N$ is the field norm from $\FF_{2^{sm}}$ to $\FF_{2^m}$. Then $\tau'$ has order~$v=7^e$ and the Davenport-Hasse theorem~\cite[Theorem 5.14]{LidNie1997} states $G(\tau')=G(\tau)^s$, so that
\[
\abs{\arg G(\tau')}\le \frac{\e}{7^{e-1}}.
\]
Now let $\chi$ be a multiplicative character of $\FF_{2^{sm}}$ of order $7^d$, where $1\le d\le e$. Then by Lemma~\ref{lem:Gauss_sum_evaluations} we have
\[
\abs{\arg G(\chi)}\le 7^{e-d}\,\abs{\arg G(\tau')},
\]
which completes the proof.
\end{proof}
\par
We need one more classical result from probabilistic combinatorics due to Spencer~\cite{Spe1985}.
\begin{lemma}[{\cite[Theorem 7]{Spe1985}}]
\label{lem:lin_forms}
Let $A$ be a matrix of size $M\times N$ satisfying $M\ge N$ with real-valued entries of absolute value at most~$1$. Then, for all sufficiently large $N$, there exists $u\in\{-1,1\}^N$ such that
\[
\norm{Au}\le11\sqrt{N\log(2M/N)},
\]
where $\norm{\cdot}$ is the maximum norm on $\RR^M$.
\end{lemma}
\par
We now prove Proposition~\ref{pro:main}.
\begin{proof}[Proof of Proposition~\ref{pro:main}]
Write $m=\ord_v(2)$. 
Lemma~\ref{lem:Gauss_sum_error} implies that, for all $\e>0$, there is an infinite set~$S$ of odd positive integers such that
\[
\Biggabs{\frac{G(\chi)}{2^{sm/2}}-1}\le \e
\]
for all $s\in S$ and all nontrivial multiplicative characters $\chi$ of $\FF_{2^{sm}}$ of order dividing $v$. Writing $n=sm$ and taking $\e=\tfrac{1}{2}\sqrt{\log (2v)/v^3}$, Lemma~\ref{lem:gauss_sum_fourier} then implies that
\[
\max_{a\in\FF_{2^n}}\;\bigabs{\hat f_2(a)}\le 1+\frac{1}{2}\sqrt{\frac{\log (2v)}{v}}
\]
for infinitely many odd positive integers $n$.
\par
It remains to consider~$\hat f_1$. Since
\[
\hat f_1(a)=\frac{1}{2^{n/2}}\sum_{y\in H}h(y)\psi(ay),
\]
we find from Lemma~\ref{lem:lin_forms} with $M=2^n$ and $N=(2^n-1)/v$ that, for all sufficiently large $n$, there exists $h:H\to\{-1,1\}$ such that 
\[
\max_{a\in\FF_{2^n}}\;\abs{\hat f_1(a)}\le 11\sqrt{\frac{\log(2v)}{v}}.
\]
Since $\hat f(a)=2^{-n/2}+\hat f_1(a)+\hat f_2(a)$ for all $a\in\FF_{2^n}$, there is an odd integer $n$ that is a multiple of $m=\ord_v(2)$ such that
\[
\max_{a\in\FF_{2^n}}\;\abs{\hat f(a)}\le 1+12\sqrt{\frac{\log (2v)}{v}},
\]
as required.
\end{proof}
\par
We now comment on the required modifications of our proof to prove Theorem~\ref{thm:pw_balanced}. The function $h$ identified in the proof of Proposition~\ref{pro:main} satisfies
\[
\Biggabs{\sum_{y\in H}h(y)}\le 11\sqrt{2^n\,\frac{\log(2v)}{v}}.
\]
Therefore we have to change at most $6\sqrt{2^n\log(2v)/v}$ values of the function~$h$ to get 
\[
\sum_{y\in H}h(y)=-1.
\]
This increases $\abs{\hat f_1(a)}$ by at most $12\sqrt{\log(2v)/v}$. The resulting function $f$ is balanced and satisfies
\[
\max_{a\in\FF_{2^n}}\;\abs{\hat f(a)}\le 1+24\sqrt{\frac{\log (2v)}{v}}.
\]
Using $1\le\mu'_{n+2}\le\mu'_n$, this shows that $\lim_{i\to\infty}\mu'_{2i+1}=1$. We combine this with $\lim_{i\to\infty}\mu'_{2i}=1$, which was already shown in~\cite{Dob1995}, but also follows from our argument using further slight modifications, to obtain a proof of Theorem~\ref{thm:pw_balanced}.

\section*{acknowledgements}
I would like to thank James A.\ Davis for many valuable discussions in the early stage of this research.


\end{document}